\documentclass[12pt,leqno]{amsart}
\usepackage{amssymb} 
\begin{document}
\theoremstyle{plain}
\newtheorem{thm}{Theorem}[section]
\newtheorem{prop}[thm]{Proposition}
\newtheorem{lem}[thm]{Lemma}
\newtheorem{clry}[thm]{Corollary}
\newtheorem{deft}[thm]{Definition}
\newtheorem{hyp}{Assumption}
\newtheorem{conjecture}[thm]{Conjecture}
\newtheorem*{KashiwaraThm}{Watermelon Theorem}

\theoremstyle{definition}
\newtheorem{rem}[thm]{Remark}
\numberwithin{equation}{section}
\newcommand{\eps}{\varepsilon}
\renewcommand{\phi}{\varphi}
\renewcommand{\d}{\partial}
\newcommand{\re}{\mathop{\rm Re} }
\newcommand{\im}{\mathop{\rm Im}}
\newcommand{\R}{\mathbf{R}}
\newcommand{\C}{\mathbf{C}}
\newcommand{\N}{\mathbf{N}} 
\newcommand{\D}{C^{\infty}_0}
\newcommand{\supp}{\mathop{\rm supp}}
\newcommand{\grad}{\mathop{\rm grad}\nolimits}
\renewcommand{\div}{\mathop{\rm div}}
\newcommand{\vol}{\mathop{\rm vol}\nolimits}
\newcommand{\WF}{\mathop{{\rm WF}_{a}}}
\title[]{On the linearized local Calder\'on problem}
\author[]{David~Dos~Santos~Ferreira, Carlos~E.~Kenig, Johannes~Sj\"ostrand, Gunther~Uhlmann}
\address{Universit\'e Paris 13 \\ Cnrs, umr 7539 LAGA \\ 99, avenue Jean-Baptiste Cl\'ement \\ F-93430 Villetaneuse \\ France}
\address{Department of Mathematics \\ University of Chicago \\ 5734 University Avenue \\ Chicago, IL 60637-1514 \\ USA}
\address{Universit\'e de Bourgogne \\ Cnrs, umr 5584 IMB \\ 9, Avenue Alain Savary, BP 47870 \\ F-21078 Dijon \\ France}
\address{Department of Mathematics \\ University of Washington \\ Seattle, WA 98195 \\ USA}
\begin{abstract}
   In this article, we investigate a density problem coming from the linearization of Calder\'on's problem with partial data. 
   More precisely, we prove that the set of products of harmonic functions on a bounded smooth domain
   $\Omega$ vanishing on any fixed closed proper subset of the boundary are dense in $L^{1}(\Omega)$ in  all dimensions
   $n \geq 2$. This is proved using ideas coming from the proof of Kashiwara's Watermelon theorem \cite{K}.
\end{abstract}
\maketitle
\setcounter{tocdepth}{1} 
\tableofcontents
%
%
\begin{section}{Introduction}
\begin{subsection}{Main results}
In the seminal article \cite{C}, A.~P.~Calder\'on asked the question of whether it is possible to determine the electrical conductivity
of a body by making current and voltage measurements at the boundary. Put in mathematical terms, the question amounts to whether the knowledge 
of the Dirichlet-to-Neumann map associated to the conductivity equation 
\begin{align}
\label{Intro:ConductEq}
     \mathop{\rm div} (\gamma \nabla u) = 0 
\end{align}
on a bounded open set  $\Omega$ with smooth boundary uniquely determines a bounded from below conductivity $\gamma \in L^{\infty}(\Omega)$. 
Using Green's formula, the problem can be reformulated in the following way: does the cancellation 
\begin{align*}
    \int\limits_{\Omega} (\gamma_{1}-\gamma_{2}) \nabla u_{1} \cdot \nabla u_{2} \, dx = 0
\end{align*}
for all solutions $u_{1},u_{2}$ in $H^1(\Omega)$ of equation \eqref{Intro:ConductEq} with respective conductivities $\gamma_{1},\gamma_{2}$  
imply that $\gamma_{1}$ and $\gamma_{2}$ are equal? Since 1980, the problem has been extensively 
studied and answers have been given in many cases (see for instance \cite{KV,SU,N,AP}). In his article \cite{C}, Calder\'on studied the linearization
of this problem at constant conductivities $\gamma=\gamma_{0}$: does the cancellation 
\begin{align*}
    \int\limits_{\Omega} \gamma \nabla u \cdot \nabla v \, dx = 0
\end{align*}
for all pairs of \textit{harmonic} functions $(u,v)$ imply that $\gamma \in L^{\infty}(\Omega)$ vanishes identically? The answer can easily be seen 
to be true by using harmonic exponentials. A similar and related inverse problem for the Schr\"odinger equation
\begin{align}
\label{Intro:SchrodEq}
     -\Delta u + q u =0 
\end{align}     
on a bounded open set with smooth boundary $\Omega$ is whether the Dirichlet-to-Neumann map associated to this equation uniquely determines the bounded potential $q$
(see for instance  \cite{SU,N,B}).  In \cite{SU}, Calder\'on's problem is reduced to this problem for $\gamma \in C^2$.
The linearization of this inverse problem at $q=0$ leads to the question of density of products of harmonic functions in $L^{1}(\Omega)$.  
Again the use of harmonic exponentials is enough to conclude this.

We are interested in local versions of these inverse problems, in particular to prove that if $\Lambda_{q_{j}}$ denotes the Dirichlet-to-Neumann map associated 
with the Schr\"odinger equation \eqref{Intro:SchrodEq} with potential $q_{j}$ and if
\begin{align}
\label{Intro:LocalPb}
    \Lambda_{q_{1}}f|_{\Sigma}=\Lambda_{q_{2}}f|_{\Sigma}, \quad \forall f \in H^{\frac{1}{2}}(\d \Omega), \quad \supp f \subset \Sigma,
\end{align}
where $\Sigma$ is an open neigbourhood of some point in the boundary, then $q_{1}=q_{2}$. An equivalent formulation is that the cancellation
\begin{align*}
     \int\limits_{\Omega} q u_{1} u_{2} \, dx = 0
\end{align*} 
for all solutions $u_{1},u_{2}$ in $H^1(\Omega)$ of the Schr\"odinger equations \eqref{Intro:SchrodEq} with bounded potentials $q_{1},q_{2}$, whose restrictions to the boundary 
are supported in $\Sigma$, imply that $q$ vanishes identically. This result has recently been proved in dimension $n=2$ by Imanuvilov,  Uhlmann,
and Yamamoto in \cite{IUY2}.  The case of partial data where one drops the support constraint on the test functions $f \in H^{\frac{1}{2}}(\d \Omega)$
was treated in various situations by Bukhgeim and Uhlmann \cite{B}, Kenig, Sj\"ostrand and Uhlmann  \cite{KSU}, Isakov \cite{I} in dimension $n \geq 3$ and 
Imanuvilov, Uhlmann, and Yamamoto \cite{IUY} in dimension $2$.
However the question of global identifiability from \eqref{Intro:LocalPb} is still open in dimension $n \geq 3$.

As a first step in this study, we consider here the linearized version of the local problem: we add the constraint that the restriction of the harmonic functions 
to the boundary vanishes on any fixed closed proper subset of the boundary. 
\begin{thm}
\label{Intro:LinThm}
    Let $\Omega$ be a connected bounded open set in $\R^{n}$, $n \geq 2$, with smooth boundary. The set of products of 
   harmonic functions in $C^{\infty}(\overline{\Omega})$ which vanish on a closed proper subset $\Gamma \subsetneq \d\Omega$ of the boundary is dense in $L^1(\Omega)$. 
\end{thm}

Another motivation for considering this linearized problem is the following possible application of Theorem  \ref{Intro:LinThm} to travel time tomography 
in dimension $2$. We conjecture that one can use Theorem \ref{Intro:LinThm}
and a method developed by Pestov and Uhlmann in \cite{PU2} to solve the corresponding global problem
to show that in a \textit{simple} 2-dimensional Riemannian manifold with boundary,  
the conformal factor of the metric is uniquely determined from partial knowledge of the boundary distance function.
A Riemannian manifold with boundary $(X,g)$ is said to be \textit{simple} if its boundary is strictly convex and if for all $x  \in \d X$,
the exponential map $\exp_x:U_{x} \to X$ is a diffeomorphism from a neighbourhood $U_{x}$ of $0$ in $T_xX$ to $X$. 
\begin{conjecture}
    Let $(X,g_{1})$ and $(X,g_{2})$ be two simple compact Riemannian manifolds of dimension $2$ with boundary, 
    and $d_{1}$ and $d_{2}$ denote their respective Riemannian distances. Let $Y$ be a non-empty open subset 
    of the boundary $\d X$ and suppose that $g_1$ and $g_2$ are conformal metrics. If 
          $$ d_{1}|_{Y \times \d X} =  d_{2}|_{Y \times \d X} $$
    then $g_{1}=g_{2}$.
\end{conjecture}
We hope to come back to this possible application in future work.
\end{subsection}

\begin{subsection}{The Watermelon approach}
The Segal-Bargmann transform of an $L^{\infty}$ function $f$ on $\R^{n}$ is given by the following formula
\begin{align*}
    Tf(z) = \int\limits_{\R^{n}} e^{-\frac{1}{2h}(z-y)^{2}} f(y) \, dy 
\end{align*}
with $z=x+i\xi \in \C^{n}$. The extension of this definition to tempered distributions is straightforward.
The Segal-Bargmann transform is related to the microlocal analysis of analytic singularities of a distribution:
the analytic wave front set $\WF(f)$ of $f$ is the complement of the set of all covectors $(x_{0},\xi_{0}) \in T^{*}\R^{n} \setminus 0$ such that there exists
a neighbourhood $V_{z_{0}}$ of $z_{0}=x_{0}-i\xi_{0}$ in $\C^{n}$,  a cutoff function $\chi \in \D(\R^{n})$ with $\chi(x_{0})=1$, and two constants 
$c>0$ and $C>0$ for which one has the estimate
\begin{align}
\label{Intro:WFdef}
   |T(\chi f)(z)| \leq  C e^{-\frac{c}{h}+\frac{1}{2h}|\im z|^{2}}, \quad \forall z \in V_{z_{0}}, \quad \forall h \in (0,1].
\end{align}
The analytic wave front set $\WF(f)$ is a closed conic set and its image by the first projection $T^{*}\R^{n} \to \R^{n}$ is the analytic singular 
support of $f$, i.e. the set of points $x_{0} \in \R^{n}$ for which there is no neighbourhood on which $f$ is a real analytic function.

When a distribution $f$ is supported on a half space $H$ and when $x_{0} \in \supp f \cap \d H$ then $f$ cannot be analytic at $x_{0}$, 
so the analytic wave front set of $f$ cannot be empty. The following result (see \cite{H}) gives explicitly covectors which are in the wave front set.
\begin{thm}
\label{Water:MicroHolmgren}
     Let $f$ be a distribution supported in a half-space $H$,  if $x_{0} \in \d H$ belongs to the support of $f$, then 
     $(x_{0},\pm \nu)$ belongs to the analytic wave front set of $f$ where $\nu$ denotes a unit conormal to the hyperplane $\d H$. 
\end{thm}
One sometimes refers to Theorem \ref{Water:MicroHolmgren} as the microlocal version of Holmgren's uniqueness theorem. This is due to the fact that the combination of
this result together with microlocal ellipticity 
     $$ \WF(u) \subset \WF(Pu) \cup \mathop{\rm char} P $$
in the conormal direction (equivalent to the fact that the hypersurface is non-characteristic) yields Holmgren's uniqueness theorem (see \cite{Sj}
chapter 8, \cite{H} chapter VIII and \cite{H2}). Other applications involve the proof of Helgason's support theorem on the Radon transform and extensions 
(see \cite{BQ} and \cite{H2}) of this result. Theorem \ref{Water:MicroHolmgren} has also proved to be a useful tool in the resolution of inverse
problems (see \cite{KSU} and \cite{DSFKSU}) with partial data. In fact the microlocal version of Holmgren's uniqueness theorem is a consequence%
\footnote{There are of course other ways to prove Theorem \ref{Water:MicroHolmgren}.}
of a more general result on the analytic wave front set due to Kashiwara (see \cite{K,Sj,H})
\begin{KashiwaraThm}
    Let $f$ be a distribution supported in a half-space $H$, if $x_{0} \in \d H$ and if $(x_{0},\xi_{0})$ belongs to the analytic wave front set of $f$, then 
    so does $(x_{0},\xi_{0}+t\nu)$ where $\nu$ denotes a unit conormal to the hyperplane $\d H$ provided $\xi_{0}+t\nu \neq 0$. 
\end{KashiwaraThm}
From Kashiwara's Watermelon theorem, it is easy to deduce the microlocal version of Holmgren's uniqueness theorem: if $f$ is supported in the half-space $H$ 
and $x_{0} \in \d H \cap \supp f$ then there exists $(x_{0},\xi_{0})$ in the analytic wave front set of $f$ since $f$ cannot be analytic at $x_{0}$, then
$(x_{0},\xi_{0}+t\nu) \in \WF(f)$ by the Watermelon theorem, which implies $(x_{0},\nu+\xi_{0}/t) \in \WF(f)$ since the wave front set is conic 
and finally $(x_{0},\nu) \in \WF(f)$ by passing to the limit since the wave front set is closed.

One possible proof of Kashiwara's Watermelon theorem involves the Segal-Bargmann transform. Note that there is an \textit{a priori} exponential bound on the Segal-Bargmann transform of an $L^{\infty}$ function
\begin{align*}
    |Tf(z)| \leq (2\pi h)^{\frac{n}{2}}e^{\frac{1}{2h}|\im z|^{2}}\|f\|_{L^{\infty}}.
\end{align*}
If $f$ is supported in the half-space $x_{1}\leq 0$ then the former estimate can be improved into
\begin{align*}
    |Tf(z)| \leq   (2\pi h)^{\frac{n}{2}}e^{\frac{1}{2h}(|\im z|^{2}-|\re z_{1}|^{2})} \|f\|_{L^{\infty}}
\end{align*}
when $\re z_1 \geq 0$. The exponent in the right-hand side is harmonic with respect to $z_{1}$.
The idea of the proof of the Watermelon theorem is to propagate the exponential decay by use of the maximum principle. If $f$ is supported in the half-space $x_{1}\leq 0$, one works with the subharmonic function
   $$ \phi(z_1)+\frac{1}{2}(\re z_1)^2-\frac{1}{2}(\im z_1)^2+h\log|Tf(z_0+z_1e_1)| $$
on a rectangle $R$. One of the edges of $R$ is contained in the neighbourhood $V_{z_0}$ where there is the additional exponential decay 
\eqref{Intro:WFdef} of the Segal-Bargmann transform and one chooses $\phi$ to be a non-negative harmonic function vanishing on the boundary of $R$
except for the segment where there is the exponential decay. The fact that $\phi$ is positive on the interior of the rectangle $R$ allows 
to propagate the exponential decay of the Segal-Bargmann transform and this translates into the propagation of singularities described in the Watermelon theorem. 
For more details we refer the reader to \cite{Sj, Sj2}. In this note, we will use a variant of this argument adapted to our problem. 
\end{subsection}
\end{section}

\subsection*{Acknowledgements}

D. DSF.~is parly supported by ANR grant Equa-disp.
C.E.K.~is partly supported by NSF grant DMS-0456583. 
G.U. would like to acknowledge partial support of NSF and a Walker Family Endowed Professorship.

%
%
\begin{section}{From local to global results}
Let $\Omega$ be a connected bounded open set in $\R^{n}$ with smooth boundary. Consider a proper closed subset $\Gamma \subsetneq \d\Omega$ of 
the boundary and a function $f \in L^{\infty}(\Omega)$. Our aim is to prove that the cancellation 
\begin{align}
\label{Local:IntCancel}
    \int\limits_{\Omega} f u v \, dx =0
\end{align}
for any pair of harmonic functions $u$ and $v$ in $C^{\infty}(\overline{\Omega})$ satisfying
     $$ u|_{\Gamma} = v|_{\Gamma} = 0  $$
implies that $f$ vanishes identically. Note that the bigger the subset $\Gamma$ is, the smaller the set of harmonic functions vanishing 
on $\Gamma$ is. Therefore we can assume that the complement of $\Gamma$ in the boundary, is a small open neighbourhood of some point of the boundary.
We will obtain Theorem \ref{Intro:LinThm} as a corollary of a local result.
\begin{thm}
\label{Local:LocLinThm}
   Let $\Omega$ be a bounded open set in $\R^{n}$, $n \geq 2$, with smooth boundary, let $x_0 \in \d\Omega$ and $\Gamma$ be the complement of an open
   boundary neighbourhood  of $x_0$. There exists $\delta>0$  such that if we have the cancellation \eqref{Local:IntCancel} 
   for any pair of harmonic functions $u$ and $v$ in $C^{\infty}(\overline{\Omega})$ vanishing on $\Gamma$, then  $f$ vanishes on $B(x_0,\delta) \cap \Omega$.
\end{thm}
Let us see how this local result implies the global one. We have learned of this technique from unpublished work of Alessandrini, Isozaki and Uhlmann
(personal communication). We will need the following approximation lemma in the spirit of the Runge approximation theorem.
\begin{lem}
\label{Local:RungeApprox}
    Let $\Omega_1 \subset \Omega_2$ be two bounded open sets with smooth boundaries. Let $G_{\Omega_2}$ be the Green kernel associated 
    to the open set $\Omega_{2}$
        $$ -\Delta_yG_{\Omega_{2}}(x,y)=\delta(x-y), \quad G_{\Omega_{2}}(x,\cdot)|_{\d \Omega_2}=0. $$
    Then the set
    \begin{align}
    \label{Local:DenseSet}
          \bigg\{ \int\limits_{\Omega_{2}} G_{\Omega_2}(\cdot,y) a(y) \, dy : \, a \in C^{\infty}(\overline{\Omega}_2), \, 
          \supp a \subset \overline{\Omega}_{2} \setminus \Omega_{1} \bigg\}
    \end{align}
    is dense for the $L^2(\Omega_{1})$ topology in the subspace of harmonic functions $u \in C^{\infty}(\overline{\Omega}_{1})$ such that
    $u|_{\d\Omega_{1}\cap \d\Omega_{2}}=0$.
\end{lem}
\begin{proof}
     Let $v \in L^2(\Omega_{1})$ be a function which is orthogonal to the subspace \eqref{Local:DenseSet}, then by Fubini we have
     \begin{align*}
          \int\limits_{\Omega_{2}} a(y) \bigg( \int\limits_{\Omega_{1}} G_{\Omega_2}(x,y) v(x) \, dx \bigg) \, dy = 0
     \end{align*}
     for all $a \in C^{\infty}(\overline{\Omega}_2)$ supported in $\overline{\Omega}_{2} \setminus \Omega_{1}$, therefore
     \begin{align*}
          \int\limits_{\Omega_{1}} G_{\Omega_2}(x,y) v(x) \, dx = 0, \quad \forall y \in \overline{\Omega}_{2} \setminus \Omega_{1}.
     \end{align*}
     We want to show that $v$ is orthogonal to any harmonic function $u \in C^{\infty}(\overline{\Omega}_{1})$ 
     such that $u|_{\d\Omega_{1}\cap \d\Omega_{2}}=0$.

     Let  $u \in C^{\infty}(\overline{\Omega}_{1})$ be a such a harmonic function. If we consider
     \begin{align*}
          w(y) = \int\limits_{\Omega_{1}} G_{\Omega_2}(x,y) v(x) \, dx \in H^2(\Omega_{2}) \cap H^1_{0}(\Omega_{2})
     \end{align*}
     then we have by Green's formula
     \begin{align*}
           \int_{\Omega_{1}} u v \, dx &= \int_{\Omega_{1}} u \Delta w \, dx - \int_{\Omega_{1}} w \Delta u \, dx \\
           &= \int_{\d \Omega_{1}} u \d_{\nu}w \, dx - \int_{\d\Omega_{1}} w \d_{\nu} u \, dx.
     \end{align*}
     Note that the trace of $w$ vanishes on $\d\Omega_{1} \cap \d \Omega_{2}$ since $w \in H^1_{0}(\Omega_{2})$, therefore we have
     \begin{align}
     \label{Local:Orthogonality}
           \int_{\Omega_{1}} u v \, dx = \int_{\d \Omega_{1} \setminus \d \Omega_{2}} u \d_{\nu}w \, dx - \int_{\d\Omega_{1} \setminus \d\Omega_{2}} w \d_{\nu} u \, dx.
     \end{align}
     At the beginning of this proof, we have shown that
           $$ w|_{\overline{\Omega}_{2} \setminus \Omega_{1}}=0 \quad \textrm{ hence also } \quad \nabla w|_{\overline{\Omega}_{2} \setminus \Omega_{1}}=0 $$
     and this implies that $w|_{\d \Omega_{1} \setminus \d \Omega_{2}}=0$ and $\d_{\nu}w|_{\d \Omega_{1} \setminus \d \Omega_{2}}=0$.
     Therefore the integral \eqref{Local:Orthogonality} vanishes and this proves that $v$ is orthogonal to any harmonic function in $C^{\infty}(\overline{\Omega}_{1})$ 
     vanishing on $\d \Omega_{1} \cap \d \Omega_{2}$.
\end{proof}
\begin{proof}[Proof of Theorem \ref{Intro:LinThm}]
   We want to prove that $f$ vanishes inside $\Omega$. We fix a point $x_1 \in \Omega$ and let $\theta:[0,1] \to \overline{\Omega}$ be a $C^1$ curve joining
   $x_0 \in \d \Omega \setminus \Gamma$ to $x_1$ such that $\theta(0)=x_{0}$, $\theta'(0)$ is the interior normal to $\d\Omega$ at $x_0$ 
   and $\theta(t) \in \Omega$ for all $t \in (0,1]$. We consider the closed neighbourhood
       $$ \Theta_{\eps}(t) = \big\{ x \in \overline{\Omega} : d\big(x,\theta([0,t])\big)\leq \eps \big\} $$
   of the curve ending at $\theta(t)$, $t\in [0,1]$ and the set
   \begin{align*}
      I=\big\{ t \in [0,1]: f \textrm{ vanishes a.e. on } \Theta_{\eps}(t) \cap \Omega \big\} 
   \end{align*}
   which is obviously a closed subset of $[0,1]$. By Theorem \ref{Local:LocLinThm} it is non-empty if $\eps$ is small enough.
   Let us prove that $I$ is open. If $t \in I$ and $\eps$ is small enough, then we may suppose 
   $\d\Theta_{\eps}(t) \cap \d\Omega \subset \d\Omega \setminus \Gamma$ and $\Omega \setminus \Theta_{\eps}(t)$ can 
   be smoothed out into an open subset $\Omega_1$ of $\Omega$ with smooth boundary such that 
      $$ \Omega_1  \supset \Omega \setminus \Theta_{\eps}(t) \quad \d\Omega \cap \d \Omega_1 \supset \Gamma. $$
   We also augment the set $\Omega$ by smoothing out the set $\Omega \cup B(x_0,\eps')$ into an open set $\Omega_2$
   with smooth boundary; if $\eps'$ is small enough then one can construct $\Omega_2$ in such a way that
      $$ \d\Omega_2 \cap \d \Omega \supset \d\Omega_1 \cap \d \Omega \supset \Gamma. $$
   
   Let $G_{\Omega_{2}}$ be the Green kernel associated to the open set $\Omega_2$
       $$ -\Delta_yG_{\Omega_{2}}(x,y)=\delta(x-y), \quad G_{\Omega_{2}}(x,\cdot)|_{\d \Omega_2}=0. $$
   The function
       $$ \int\limits_{\Omega_1} f \, G_{\Omega_{2}}(x,y) \, G_{\Omega_{2}}(t,y) \, dy, \quad 
          t,x \in \Omega_2 \setminus \overline{\Omega}_1 $$
   is harmonic (both as a function of the $t$ and $x$ variables) and satisfies
       $$  \int\limits_{\Omega_1} f \, G_{\Omega_{2}}(x,y) \, G_{\Omega_{2}}(t,y) \, dy = 
           \int\limits_{\Omega} f \, G_{\Omega_{2}}(x,y) \, G_{\Omega_{2}}(t,y) \, dy $$
   since $f$ vanishes on $\Theta_{\eps}(t) \cap \Omega$. When $t,x$ belong to $\Omega_2 \setminus \overline{\Omega}$, 
   this integral is $0$ since the Green functions are $C^{\infty}(\overline{\Omega})$, harmonic on $\Omega$ 
   and vanish on $\Gamma \subset \d\Omega_2$. By unique continuation and continuity, we have
   \begin{align}
       \int\limits_{\Omega_1} f \, G_{\Omega_{2}}(x,y) \, G_{\Omega_{2}}(t,y) \, dy= 0,  \quad 
       t,x \in \overline{\Omega}_2 \setminus \Omega_1.
   \end{align}
   By Fubini, this means that we will have $\int_{\Omega_1} f u v \, dx = 0$
   for all functions $u,v$ on $\Omega_1$ belonging to the subspace \eqref{Local:DenseSet}. By continuity of
   the bilinear form
   \begin{align*}
       L^2(\Omega_1) \times L^2(\Omega_1) &\to \C \\
       (u,v) &\mapsto \int\nolimits_{\Omega_{1}}f u v \, dx
   \end{align*}
   and by Lemma \ref{Local:RungeApprox}, we have
   \begin{align}
   \label{Local:CancelIntSmaller}
       \int\limits_{\Omega_1} f u v \, dx = 0
   \end{align}
   for all functions $u,v$ in $C^{\infty}(\overline{\Omega}_1)$ harmonic on $\Omega_1$ which vanish on  $\d\Omega_{1} \cap \d \Omega_{2}$.
   
      Thanks to Theorem \ref{Local:LocLinThm}, the cancellation \eqref{Local:CancelIntSmaller} implies that $f$ vanishes 
   on a neighbourhood of $\d\Omega_1 \setminus (\d\Omega_{1} \cap \d \Omega_{2})$. This shows that $f$ vanishes on a slightly bigger neighbourhood 
   $\Theta_{\eps}(\tau), \tau>t$ of the curve, hence that $I$ is an open set. By connectivity, we conclude that $I=[0,1]$ and therefore 
   that $x_1 \notin \supp f$. Since the choice of $x_1$ is arbitrary, this completes the proof of Theorem \ref{Intro:LinThm}.
\end{proof}
\end{section}
%
%
\begin{section}{Harmonic exponentials}

This section and the next are devoted to the proof of Theorem~\ref{Local:LocLinThm}. 
One can suppose that $\Omega \setminus \{x_0\}$ is on one side of the tangent hyperplane $T_{x_0}(\Omega)$ at $x_0$ by making a conformal
transformation. Pick $a \in \R^{n} \setminus \overline{\Omega}$ on the line segment in the direction of the outward normal to $\d \Omega$ at $x_0$, 
then there is a ball $B(a,r)$ such that $\d B(a,r) \cap \overline{\Omega} =\{x_0\}$, and there is a conformal transformation 
\begin{align*}
   \psi : \R^n \setminus B(a,r) &\to \overline{B(a,r)} \\
   x &\mapsto \frac{x-a}{|x-a|^2}r^2+a
\end{align*}
which fixes $x_0$ and exchanges the interior and the exterior of the ball $B(a,r)$. The hyperplane $H:(x-x_0) \cdot (a-x_0)=0$
is tangent to $\psi(\Omega)$, and the image $\psi(\Omega) \setminus  \{x_0\}$ by the conformal transformation lies inside the ball $B(a,r)$,
therefore on one side of $H$.  The fact that functions are supported on the boundary close to $x_0$ is left unchanged.
Since a function is harmonic on $\Omega$ if and only if its Kelvin transform
    $$ u^* = r^{n-2}|x-a|^{-n+2} u \circ \psi $$
is harmonic on $\psi(\Omega)$, \eqref{Local:IntCancel} becomes
\begin{align*}
    0=\int\limits_{\Omega} f u v \, dx =  \int\limits_{\psi(\Omega)} r^4 |x-a|^{-4} f\circ \psi  \, u^* v^* \, dx 
\end{align*}
for all harmonic functions $u^*,v^*$ on $\psi(\Omega)$. 
If $|x-a|^{-4} f \circ \psi$ vanishes close to $x_0$ then so does $f$.
Moreover, by scaling one can assume that $\Omega$ is contained in a ball of radius $1$.  

Our setting will therefore be as follows: $x_0=0$, the tangent hyperplane at $x_{0}$ is given by $x_1=0$ and 
\begin{align}
     \Omega \subset \big\{x \in \R^{n} : |x+e_{1}| < 1\}, \quad \Gamma = \big\{x \in \d\Omega : x_{1} \leq -2c \big \}. 
\end{align}
The prime will be used to denote the last $n-1$ variables so that $x=(x_1,x')$ for instance.
The Laplacian on $\R^{n}$ has $p(\xi)=\xi^2$ as a principal symbol, we denote by $p(\zeta)=\zeta^{2}$ the continuation 
of this principal symbol on $\C^n$, we consider
    $$ p^{-1}(0) = \big\{ \zeta \in \C^n : \zeta^2=0\big\}. $$
In dimension $n=2$, this set is the union of two complex lines
    $$ p^{-1}(0) = \C \gamma \cup \C \overline{\gamma} $$
where $\gamma=ie_{1}+e_{2}=(i,1) \in \C^{2}$. Note that $(\gamma,\overline{\gamma})$ is a basis of $\C^{2}$: the decomposition of a complex vector 
in this basis reads
\begin{align}
\label{Harm:decomp2D}
    \zeta =\zeta_{1} e_{1} + \zeta_{2} e_{2}=\frac{\zeta_{2}-i\zeta_{1}}{2} \, \gamma + \frac{\zeta_{2}+i\zeta_{1}}{2} \, \overline{\gamma}.
\end{align}  
Similarly for $n\geq 2$, the differential of the map
\begin{align*}
   s : p^{-1}(0) \times p^{-1}(0) &\to \C^n \\
   (\zeta,\eta) &\mapsto \zeta + \eta
\end{align*}
at $(\zeta_0,\eta_0)$ is surjective
\begin{align*}
   Ds(\zeta_0,\eta_0) : T_{\zeta_0}p^{-1}(0) \times T_{\eta_0}p^{-1}(0) &\to \C^n \\
   (\zeta,\eta) &\mapsto \zeta + \eta
\end{align*}
provided  $\C^n=T_{\zeta_0}p^{-1}(0) + T_{\eta_0}p^{-1}(0)$, i.e. provided $\zeta_0$ and $\eta_0$ are linearly independent.
In particular, this is the case if $\zeta_0=\gamma$ and $\eta_0=-\overline{\gamma}$; as a consequence all $z \in \C^n$, $|z-2ie_1|<2\eps$ 
may be decomposed as a sum  of the form
\begin{align}
\label{Harm:decomp}
   z=\zeta + \eta, \quad \textrm{ with } \zeta,\eta \in p^{-1}(0), \; |\zeta-\gamma|<C\eps, \; |\eta+\overline{\gamma}|<C\eps. 
\end{align}
provided $\eps>0$ is small enough. 

The exponentials with linear weights
    $$ e^{-\frac{i}{h} x \cdot \zeta}, \quad \zeta \in p^{-1}(0) $$
are harmonic functions. We need to add a correction term in order to obtain harmonic functions $u$ satisfying the boundary requirement 
$u|_{\Gamma}=0$. Let $\chi \in \D(\R^{n})$ be a cutoff function which equals $1$ on $\Gamma$, we consider the solution
$w$ to the Dirichlet problem
\begin{align}
    \left\{ 
    \begin{aligned}
         \Delta w &=0 \quad \textrm{ in } \Omega \\
         w|_{\d \Omega} &= -(e^{-\frac{i}{h} x \cdot \zeta}\chi )|_{\d\Omega}.
    \end{aligned}
    \right.
\end{align} 
The function
\begin{align*}
    u(x,\zeta) =  e^{-\frac{i}{h} x \cdot \zeta} + w(x,\zeta)
\end{align*}
is  in $C^{\infty}(\overline{\Omega})$, harmonic and satisfies $u|_{\Gamma}=0$. We have the following bound on $w$:
\begin{align}
\label{Harm:Remainder}
     \|w\|_{H^{1}(\Omega)} &\leq C_{1} \|e^{-\frac{i}{h} x \cdot \zeta} \chi \|_{H^{\frac{1}{2}}(\d\Omega)} \\ \nonumber
     &\leq C_{2} (1+h^{-1} |\zeta|)^{\frac{1}{2}} \, e^{\frac{1}{h} H_{K}(\im \zeta)}
\end{align}
where $H_{K}$ is the supporting function of the compact subset $K=\supp \chi \cap \d\Omega$ of the boundary
    $$ H_{K}(\xi) = \sup_{x \in K} x \cdot \xi , \quad \xi \in \R^{n}. $$
In particular, if we take $\chi$ to be supported in $x_{1}\leq-c$ and equal to $1$ on $x_{1}\leq-2c$
then the bound \eqref{Harm:Remainder} becomes
\begin{align}
\label{Harm:wBound}
     \|w\|_{H^{1}(\Omega)} \leq C_{2}(1+h^{-1} |\zeta|)^{\frac{1}{2}} \, e^{-\frac{c}{h} \im \zeta_{1}} \, e^{\frac{1}{h}|\im \zeta'|}
     \quad \textrm{ when } \im \zeta_{1} \geq 0.
\end{align}

Our starting point is the cancellation of the integral
\begin{align}
\label{Harm:CancelInt} 
   \int\limits_{\Omega} f(x) u(x,\zeta) u(x,\eta) \, dx =0, \quad \zeta,\eta \in p^{-1}(0)
\end{align}
which may be rewritten under the form
\begin{multline*}
      \int\limits_{\Omega} f(x)  e^{-\frac{i}{h} x \cdot (\zeta+\eta)} \, dx = -  \int\limits_{\Omega} f(x)  e^{-\frac{i}{h} x \cdot \zeta} w(x,\eta) \, dx \\
       -  \int\limits_{\Omega} f(x)  e^{-\frac{i}{h} x \cdot \eta} w(x,\zeta) \, dx - \int\limits_{\Omega} f(x)  w(x,\zeta) w(x,\eta) \, dx.
\end{multline*}
This allows us to give a bound on the left-hand side term 
\begin{multline*}
      \bigg| \int\limits_{\Omega} f(x)  e^{-\frac{i}{h} x \cdot (\zeta+\eta)} \, dx \bigg| \leq 
      \|f\|_{L^{\infty}(\Omega)} \big(\|e^{-\frac{i}{h} x \cdot \zeta}\|_{L^2(\Omega)}\|w(x,\eta)\|_{L^2(\Omega)} \\ 
      +\|e^{-\frac{i}{h} x \cdot \eta}\|_{L^2(\Omega)} \|w(x,\zeta)\|_{L^2(\Omega)} + \|w(x,\eta)\|_{L^2(\Omega)}\|w(x,\zeta)\|_{L^2(\Omega)}\big).
\end{multline*}
Thus using \eqref{Harm:wBound} 
\begin{multline*}
      \bigg| \int\limits_{\Omega} f(x)  e^{-\frac{i}{h} x \cdot (\zeta+\eta)} \, dx \bigg| \leq C_3 \|f\|_{L^{\infty}(\Omega)} 
      (1+h^{-1}|\eta|)^{\frac{1}{2}} (1+h^{-1}|\zeta|)^{\frac{1}{2}} \\
      \times e^{-\frac{c}{h} \min(\im \zeta_1,\im \eta_1)} \,e^{\frac{1}{h}(|\im \zeta'|+|\im \eta'|)}
\end{multline*}
when $\im \zeta_1 \geq 0, \im \eta_1\geq 0$ and $\zeta, \eta \in p^{-1}(0)$. In particular if $|\zeta-a\gamma|<C\eps a$ 
and $|\eta+a\overline{\gamma}|<C\eps a$ with $\eps\leq 1/2C$ then
\begin{align*}
      \bigg| \int\limits_{\Omega} f(x)  e^{-\frac{i}{h} x \cdot (\zeta+\eta)} \, dx \bigg| \leq C_4 h^{-1}\|f\|_{L^{\infty}(\Omega)} 
       e^{-\frac{ca}{2h}} \, e^{\frac{2C\eps a}{h}}.
\end{align*}
Take $z \in \C^{n}$ with $|z-2ae_1|< 2\eps a$ with $\eps$ small enough, once rescaled the decomposition
\eqref{Harm:decomp} gives
   $$ z = \zeta+\eta, \quad \zeta,\eta \in p^{-1}(0), \; |\zeta-a\gamma| < C\eps a, \: |\eta+a\overline{\gamma}| < C\eps a $$
we therefore get the estimate
\begin{align}
\label{Harm:EstFourier}
      \bigg| \int\limits_{\Omega} f(x)  e^{-\frac{i}{h} x \cdot z} \, dx \bigg| \leq C_4 h^{-1}\|f\|_{L^{\infty}(\Omega)} 
       e^{-\frac{ca}{2h}} \, e^{\frac{2C\eps a}{h}}.
\end{align}
for all $z \in \C^n$ such that $|z-2ae_1|< 2\eps a$.

In order to conclude, one needs to extrapolate the exponential decay to more values of the frequency variable $z$. 
This will be achieved using a variant of the proof of the Watermelon theorem. We extend the function $f$ to $\R^n$ by assigning to it 
the value $0$ outside $\Omega$.
\end{section}
%
%
\begin{section}{A watermelon approach}
Let us recall the definition of the Segal-Bargmann transform of an $L^{\infty}$ function $f$ on $\R^{n}$ 
\begin{align*}
    Tf(z) = \int\limits_{\R^{n}} e^{-\frac{1}{2h}(z-y)^{2}} f(y) \, dy,  \quad  z \in \C^{n}
\end{align*}
and the \textit{a priori} exponential bound 
\begin{align}
\label{Water:PlainEstBI}
    |Tf(z)| \leq (2\pi h)^{\frac{n}{2}}e^{\frac{1}{2h}|\im z|^{2}}\|f\|_{L^{\infty}}.
\end{align}
If $f$ is supported in the half-space $x_{1}\leq 0$ then the former estimate can be improved into
\begin{align}
\label{Water:EstFBI}
    |Tf(z)| \leq   (2\pi h)^{\frac{n}{2}}e^{\frac{1}{2h}(|\im z|^{2}-|\re z_{1}|^{2})} \|f\|_{L^{\infty}}
\end{align}
when $\re z_1 \geq 0$. 

The kernel of the Segal-Bargmann transform of a function $f \in L^{\infty}$ can be written as a linear superposition of exponentials with linear weights
\begin{align*}
    e^{-\frac{1}{2h}(z-y)^{2}} = e^{-\frac{z^{2}}{2h}} (2\pi h)^{-\frac{n}{2}}  \int e^{-\frac{t^{2}}{2h}} e^{-\frac{i}{h} y \cdot (t+iz)} \, dt 
\end{align*}
therefore we get
\begin{align}
\label{Water:Superp}
    Tf(z) = (2\pi h)^{-\frac{n}{2}} \iint e^{-\frac{1}{2h}(z^{2}+t^{2})} e^{-\frac{i}{h} y \cdot (t+iz)} f(y) \, dt \, dy.
\end{align}
Suppose now that the function $f$ is supported in $\Omega$ and satisfies \eqref{Harm:CancelInt},
formula \eqref{Water:Superp} allows us to improve the estimate~\eqref{Water:EstFBI}:
\begin{align*}
    |Tf(z)|  &\leq  (2\pi h)^{-\frac{n}{2}} \int  e^{\frac{1}{2h}(|\im z|^{2}-|\re z|^{2}-t^{2})} 
    \bigg| \int e^{-\frac{i}{h} y \cdot (t+iz)} f(y) \, dy \bigg| \, dt. 
\end{align*}
Suppose now that $\re z_{1} \geq 0$: if we split the integral with respect to the variable $t$ in two integrals
\begin{multline*}
     |Tf(z)|  \leq   \frac{e^{\frac{1}{2h}(|\im z|^{2}-|\re z|^{2})}}{(2\pi h)^{\frac{n}{2}}} 
     \bigg(\int\limits_{|t| \leq \eps a}  e^{-\frac{t^{2}}{2h}} \bigg| \int e^{-\frac{i}{h} y \cdot (t+iz)} f(y) \, dy\bigg| \, dt \\
     +  \int\limits_{|t| \geq \eps a}  e^{-\frac{t^{2}}{2h}} \bigg| \int e^{-\frac{i}{h} y \cdot (t+iz)} f(y) \, dy\bigg| \, dt \bigg)
\end{multline*}
this implies 
\begin{multline}
\label{Water:ImprEstFBI}
     |Tf(z)|  \leq e^{\frac{1}{2h}(|\im z|^{2}-|\re z|^{2})} \bigg(
     \sup_{|t| \leq \eps a}   \bigg| \int e^{-\frac{i}{h} y \cdot (t+iz)} f(y) \, dy\bigg| \\ 
     +  \sqrt{2} \, e^{\frac{1}{h}|\re z'|} \, e^{-\frac{\eps^{2}a^2}{4h}}
     \int\nolimits_{\Omega}  |f(y)| \, dy\bigg)
\end{multline}
since $f$ is supported in $\Omega \subset \{y_{1}\leq 0\}$. If we assume $|z-2ae_1|< \eps a$ with $\eps$ small enough%
    \footnote{Note that in dimension $n=2$ the decomposition \eqref{Harm:decomp} for $t+iz$ is explicit 
              \begin{align*}
                 t+iz = \underbrace{\frac{1}{2}\big(t_{2}-it_{1}+iz_{2}+z_{1}\big) \, \gamma}_{=\zeta}
                 +\underbrace{\frac{1}{2}\big(t_{2}+it_{1}+iz_{2}-z_{1}\big) \, \overline{\gamma}}_{=\eta}.
              \end{align*}},
the estimate \eqref{Harm:EstFourier} reads in our context
\begin{align}
\label{Water:EstFourier}
      \bigg| \int\limits_{\Omega} f(y)  e^{-\frac{i}{h} y \cdot (t+iz)} \, dy \bigg| \leq C_4 h^{-1} \|f\|_{L^{\infty}(\Omega)} 
      e^{-\frac{ca}{2h}} \, e^{\frac{2C\eps a}{h}}
\end{align}
when $|t|\leq \eps a$ and $|z-2ae_1|< \eps a$. Thus combining the two estimates \eqref{Water:EstFourier} 
and \eqref{Water:ImprEstFBI} we get 
\begin{align*}
     |Tf(z)|  \leq C_5  h^{-1} \|f\|_{L^{\infty}(\Omega)} e^{\frac{1}{2h}(|\im z|^{2}-|\re z|^{2})}  
     \,\big(e^{-\frac{ca}{2h}}e^{\frac{2C\eps a}{h}} + e^{-\frac{\eps^{2}a^2}{4h}}e^{\frac{\eps a}{h}}\big)
\end{align*}
provided $|z-2ae_1|< \eps a$. Now choosing $\eps<c/8C$ and $a>(c+4\eps)/\eps^2$ we finally obtain the bound  
\begin{align}
\label{Water:ImprEstFBIbis}
     |Tf(z)|  \leq 2C_5  h^{-1}\|f\|_{L^{\infty}(\Omega)} e^{\frac{1}{2h}(|\im z|^{2}-|\re z|^{2}-\frac{ca}{2})}.
\end{align}
To sum-up we have obtained the following bounds on the Segal-Bargmann transform of $f$
\begin{multline}
    e^{-\frac{\Phi(z_1)}{2h}} |Tf(z_1,x')| \leq \\ C h^{-1} \|f\|_{L^{\infty}(\Omega)} 
    \begin{cases}
       1 & \textrm{ when } z_1 \in \C  \\
       e^{-\frac{ca}{4h}} & \textrm{ when } |z_1-2a| \leq \frac{\eps a}{2}, \; |x'| < \frac{\eps a}{2}
    \end{cases}
\end{multline}
and when $x' \in \R^{n-1}$, where the weight $\Phi$ is given by the following expression
\begin{align*}
\Phi(z_1) = 
    \begin{cases}
       |\im z_1|^{2} & \textrm{ when } \re z_1 \leq 0  \\
       |\im z_1|^{2}-|\re z_{1}|^{2} & \textrm{ when } \re z_1\geq 0.
    \end{cases}
\end{align*}
These estimates correspond to \eqref{Water:PlainEstBI}, \eqref{Water:EstFBI} and \eqref{Water:ImprEstFBIbis}.
\begin{lem}
   Let $F$ be an entire function satisfying the following bounds
   \begin{align*}
      e^{-\frac{\Phi(s)}{2h}}|F(s)| \leq 
      \begin{cases}
         1 & \textrm{ when } s \in \C \\
         e^{-\frac{c}{2h}} & \textrm{ when } |s-L| \leq b
      \end{cases} 
   \end{align*}
   then for all $r \geq 0$ there exist $c',\delta >0$ such that $F$ satisfies
       $$ |F(s)| \leq e^{-\frac{c'}{2h}}, \quad \textrm{ when } |\re s|\leq \delta \textrm{ and } |\im s| \leq r. $$
\end{lem}
\begin{proof}
   We consider the subharmonic function
     $$ f(s) = 2h \log|F(s)|-(\im s)^2+(\re s)^2 $$
   which satisfies the bounds
   \begin{align}
      f(s) \leq 
      \begin{cases}
         (\re s)^{2} & \textrm{ when } \re s \leq 0 \\
         0 & \textrm{ when } \re s \geq 0 \\ 
         -c & \textrm{ when } |s-L| \leq b.
      \end{cases} 
   \end{align}
   We will work on the semi-disc of centre $-2\delta$ and large enough radius $R$, with cut-diameter 
   along the vertical axis $\re s = -2\delta$ and with the smaller disc of centre $L$ and radius $b$ removed from that semi-disc 
       $$ U_{\delta} = D(-\delta,R) \cap \{ \re s > -\delta\} \setminus \overline{D(L,b)}. $$ 
   We consider the harmonic function $\phi$ on $U_{\delta}$ with the following boundary values:
   \begin{itemize}
       \item[$\diamond$] $\phi=4\delta^2$ on the boundary of the semi-disc, 
       \item[$\diamond$] $\phi=-c$ on the circle of centre $L$ and radius $b$.
   \end{itemize}
   The function $\tilde{\phi}=4\delta^2-\phi$ is harmonic and non-negative on $U_{\delta}$ and attains its minimum everywhere on the cut-diameter
   of the semi-disc. By the Hopf boundary lemma, if $\nu$ stands for the interior normal, we have%
   \footnote{The radius $R$ is chosen large enough so that the points of the boundary with $|\im s| \leq r$ stay far enough from the corners
   where the Hopf lemma is no longer valid.}
   \begin{align*}
       \frac{\d \tilde{\phi}}{\d \nu} (-2\delta+iy) \geq \frac{C}{\delta} \tilde{\phi}(iy) > 0, \quad |y| \leq r <R
   \end{align*}
   where $C$ is a universal constant. By Harnack's inequality, $ \tilde{\phi}(iy)$ and
        $$ \tilde{\phi}(L-b-\delta^{2})=-c + \mathcal{O}(\delta^2) $$
   are comparable, and the constants are uniform with respect to $\delta$. Thus if $\delta$ is small enough, we get
   \begin{align*}
       -\frac{\d \phi}{\d \nu} (-2\delta+iy) \geq \frac{2c'}{\delta}, \quad |y|<r.
   \end{align*}
   From this inequality and elliptic regularity we get that
   \begin{align}
   \label{Watermelon:NegPhi}
       \phi(s) \leq  -c', \quad |\re s|\leq \delta, \, |\im s| \leq r
   \end{align}    
   if $\delta$ is small enough.
   
   We have
       $$ (f-\phi)|_{\d U_{\delta}} \leq 0. $$
   therefore by the maximum principle, the subharmonic function $f-\phi$ is non-positive on $U_{\delta}$. But according to \eqref{Watermelon:NegPhi},
   when $|\re s|\leq \delta$ and $|\im s|\leq r$ we have
   \begin{align}
      f \leq \phi \leq -c'
   \end{align}
   Therefore we have proved
       $$  e^{-\frac{\Phi(s)}{2h}}|F(s)| \leq e^{-\frac{c'}{2h}} $$
   if $|\re s|\leq \delta$ and $|\im s|\leq r$.
\end{proof}

Applying the former lemma to the function
   $$ F(s) =  \frac{h|Tf(s,x_2)|}{C  \|f\|_{L^{\infty}(\Omega)}} $$
we obtain in particular that
\begin{align*}
   |Tf(x)| \leq C  h^{-1}\|f\|_{L^{\infty}(\Omega)}e^{-\frac{c'}{2h}}
\end{align*}
for all $x \in \Omega, |x_1| \leq \delta $, provided $\delta$ has been chosen small enough. Multiplying by $(2\pi h)^{-n/2}$ and letting $h$ tend to $0$ 
we deduce
   $$ f(x) = 0, \quad \forall x \in \Omega, \quad 0\geq x_1 \geq -\delta. $$
This completes the proof of Theorem \ref{Local:LocLinThm}.
\end{section}
%
%

%
%

\begin{thebibliography}{maximal} %
%
\bibitem{AP} K.~Astala, L.~P\"aiv\"arinta, \textit{Calder{\'o}n's inverse conductivity problem in the plane}, Ann. of Math., \textbf{163} (2006), 265--299.
\bibitem{BrU} R.~M.~Brown, G.~Uhlmann, \textit{Uniqueness in the inverse conductivity problem for nonsmooth conductivities in two dimensions}, Comm. Partial Differential Equations, \textbf{22} (1997), 1009--1027.
\bibitem{BQ} J.~Boman, T.~Quinto, \textit{Support theorems for real-analytic Radon transforms}, Duke Math. J., \textbf{55} (1987),
 943--948.
\bibitem{B} A.~Bukhgeim, \textit{Recovering the potential from Cauchy data in two dimensions}, J. Inverse Ill-Posed Probl., \textbf{16} (2008), 19--34. 
\bibitem{BU} A.~Bukhgeim, G.Uhlmann, \textit{Recovering a potential from partial Cauchy data}, Comm. Partial Differential Equations, \textbf{27} (2002), 653Ð668. 
\bibitem{C} A.~P.~Calder\'on, \textit{On an inverse boundary value problem}, Seminar on Numerical Analysis and its Applications to Continuum Physics, 
Rio de Janeiro, Sociedade Brasileira de Matematica, (1980), 65--73.
\bibitem{DS} M.~Dimassi, J.~Sj\"ostrand, \textit{Spectral asymptotics in the semi-classical limit}, Cambridge University Press, 1999.
\bibitem{DSFKSU} D.~Dos Santos Ferreira, C.~E.~Kenig, J.~Sj\"ostrand, G.~Uhlmann,  \textit{Determining a magnetic Schr\"odinger operator
   from partial Cauchy data}, Comm. Math. Phys., \textbf{271} (2007), 467--488.
\bibitem{H} L.~H\"ormander, \textit{The analysis of linear partial differential operators I}, Springer-Verlag, 1985.
\bibitem{H2} L.~H\"ormander, \textit{Remarks on Holmgren's uniqueness theorem}, Ann. Inst. Fourier, \textbf{43} (1993), 1223--1251. 
\bibitem{IUY} O.~Imanuvilov,  G.~Uhlmann, M.~Yamamoto, \textit{Partial data for the Calder\'on problem in two dimensions}, preprint arXiv:0809.3037 (2008). 
\bibitem{IUY2} O.~Imanuvilov,  G.~Uhlmann, M.~Yamamoto, \textit{Global uniqueness from partial Cauchy data in two dimensions}, preprint arXiv:0810.2286 (2008). 
\bibitem{I} V.~Isakov, \textit{On uniqueness in the inverse conductivity problem with local data}, Inverse Problems and Imaging, \textbf{1} (2007), 95--105.
\bibitem{K} M.~Kashiwara, \textit{On the structure of hyperfunctions}, Sagaku no Ayumi, 15 (1970), 19--72 (in Japanese).
\bibitem{KSU} C.~E.~Kenig, J.~Sj\"ostrand, G.~Uhlmann,  \textit{The Calder\'on problem with partial data}, Ann. of Math., \textbf{165} (2007), 567--591.
\bibitem{KV} R.~Kohn, M.~Vogelius, \textit{ Determining conductivity by boundary measurements}, Comm. Pure Appl. Math., \textbf{37} (1984), 289--298. 
\bibitem{Mi} R.~Michel, \textit{Sur la rigidit\'e impos\'ee par la longueur des g\'eod\'esiques}, Invent. Math., \textbf{65} (1981) 71--83.
\bibitem{M} R.~G.~Mukhometov, \textit{The reconstruction problem of a two-dimensional Riemannian metric, and integral geometry} (Russian), Dokl. Akad. Nauk SSSR, \textbf{232} (1977), 32--35.
\bibitem{N} A.~Nachman, \textit{Global uniqueness for a two-dimensional inverse boundary value problem}, Ann. of Math., \textbf{143} (1996), 71--96.
\bibitem{PU} L.~Pestov, G.~Uhlmann, \textit{Two dimensional simple manifolds are boundary rigid}, Ann. of Math., \textbf{161} (2005), 1089--1106.
\bibitem{PU2} L.~Pestov, G.~Uhlmann, \textit{The scattering relation and the Dirichlet-to-Neumann map}, Contemp. Math., \textbf{412} (2006), 249--262. 
\bibitem{Sj} J.~Sj\"ostrand, \textit{Singularit\'es analytiques microlocales}, Ast\'erisque, 1985.
\bibitem{Sj2} J.~Sj\"ostrand, \textit{Remark on extensions of the Watermelon theorem}, Math. Res. Lett., \textbf{1}(1994), 309--317.
\bibitem{SU} J.~Sylvester, G.~Uhlmann, \textit{A global uniqueness theorem for an inverse boundary value problem}, Ann. of Math., \textbf{125} 
   (1987), 153--169.
%
\end{thebibliography}
\end{document}